\documentclass{article}
\usepackage{amssymb,amsmath,amsthm,mathrsfs,tikz-cd}
\usepackage[title]{appendix}

\title{The rational torsion subgroup of $J_0(\mathfrak{p}^r)$}

\author{Sheng-Yang Kevin Ho}

\newcommand{\Addresses}{{% additional braces for segregating \footnotesize

  \footnotesize

  \textsc{Department of Mathematics, The Pennsylvania State University, State College, PA 16801, USA}\par\nopagebreak
  \textit{E-mail address}: \texttt{kevinho@psu.edu}
}}

\date{}

\newcommand{\GL}{\operatorname{GL}}
\newcommand{\PGL}{\operatorname{PGL}}
\newcommand{\Div}{\operatorname{Div}}
\newcommand{\Tr}{\operatorname{Tr}}
\newcommand{\divisor}{\operatorname{div}}
\newcommand{\cusp}{\operatorname{cusp}}
\newcommand{\ord}{\operatorname{ord}}

\newcommand{\CC}{\mathcal{C}}
\newcommand{\TT}{\mathcal{T}}
\newcommand{\Tree}{\mathscr{T}}
\newcommand{\p}{\mathfrak{p}}
\newcommand{\q}{\mathfrak{q}}
\newcommand{\n}{\mathfrak{n}}
\newcommand{\m}{\mathfrak{m}}
\newcommand{\HarZ}{\mathcal{H}(\Tree,\mathbb{Z})}
\newcommand{\HarQ}{\mathcal{H}(\Tree,\mathbb{Q})}

\theoremstyle{plain}
\newtheorem{thm}{Theorem}[section]
\newtheorem{lem}[thm]{Lemma}
\newtheorem{conj}[thm]{Conjecture}

\newtheorem{prop}[thm]{Proposition}

\theoremstyle{definition}
\newtheorem{defn}[thm]{Definition}

\theoremstyle{remark}
\newtheorem{rem}[thm]{Remark}
\newtheorem{Ex}{Example}

\newcommand\restr[2]{{% we make the whole thing an ordinary symbol
  \left.\kern-\nulldelimiterspace % automatically resize the bar with \right
  #1 % the function
  \vphantom{\big|} % pretend it's a little taller at normal size
  \right|_{#2} % this is the delimiter
  }}

\NewDocumentCommand{\tens}{e{_^}}{%
  \mathbin{\mathop{\otimes}\displaylimits
    \IfValueT{#1}{_{#1}}
    \IfValueT{#2}{^{#2}}
  }%
}

\begin{document}
\maketitle
\Addresses

\begin{abstract}
Let $\mathfrak{n}=\mathfrak{p}^r$ be a prime power ideal of $\mathbb{F}_q[T]$ with $r\geq 2$. We study the rational torsion subgroup $\mathcal{T}(\mathfrak{p}^r)$ of the Drinfeld modular Jacobian $J_0(\mathfrak{p}^r)$. We prove that the prime-to-$q(q-1)$ part of $\mathcal{T}(\mathfrak{p}^r)$ is equal to that of the rational cuspidal divisor class group $\mathcal{C}(\mathfrak{p}^r)$ of the Drinfeld modular curve $X_0(\mathfrak{p}^r)$. As we completely computed the structure of $\mathcal{C}(\mathfrak{p}^r)$, it also determines the structure of the prime-to-$q(q-1)$ part of $\mathcal{T}(\mathfrak{p}^r)$.
\end{abstract}

\section{Introduction}
For a positive integer $N$, let $J_0(N)$ be the Jacobian variety of the classical modular curve $X_0(N)$ and $\TT(N):=J_0(N)(\mathbb{Q})_{\text{tors}}$ its rational torsion subgroup. By the Mordell-Weil theorem, $\TT(N)$ is a finite abelian group. Let $\CC(N)$ be the rational cuspidal divisor class group of $X_0(N)$ defined in \cite[p. 280]{YOO_rational_2023}. By a theorem of Manin and Drinfeld, we have $$\CC(N)\subseteq \TT(N).$$ In the early 1970s, for any prime $p$, Ogg \cite[Conjecture 2]{Ogg} conjectured that $\CC(p)=\TT(p)$ and computed that $\CC(p)$ is a cyclic group of order $\frac{p-1}{(p-1, 12)}$. Later in 1977, Mazur \cite[Theorem (1)]{mazur_modular_1977} proved this conjecture by studying the Eisenstein ideal of the Hecke algebra of level $p$. There is a generalized Ogg's conjecture:
\begin{conj} (still open) For any positive integer $N$, $$\CC(N) = \TT(N).$$
\end{conj}

In \cite[Theorem 1.7]{YOO_rational_2023}, Yoo et al. completely determined the structure of $\CC(N)$ for arbitrary positive $N$. Consequently, if the above conjecture holds true, then the structure of $\TT(N)$ is determined as well. The most notable current achievement in this regard is as follows:

\begin{thm} [Yoo {\cite[Theorem 1.4]{YOO_torsion_2023}}, Mazur, Lorenzini, Ling, Ohta, Ren, et al.]
Let $N$ be a positive integer, and let $\ell$ be any odd prime whose square does not divide $N$. Suppose that $\ell \geq 5$, it holds that
$$\CC(N)_{\ell} = \TT(N)_{\ell}.$$
Furthermore, for $\ell = 3$, the equality above remains valid under the additional condition that either $N$ is not divisible by $3$, or $N$ has a prime divisor congruent to $-1$ modulo $3$.
\end{thm}

\begin{rem}
In \cite[(1.1)]{YOO_torsion_2023}, in a particular case, when $N = p^r$ is a prime power with $r\geq 2$, it has been established that $\CC(p^r)_{\ell} = \TT(p^r)_{\ell}$ holds for any prime $\ell\nmid 2p$. The condition on $\ell$ here is less restrictive compared to the general case.
\end{rem}

In the function field analogue, let $\mathbb{F}_q$ be a finite field of characteristic $p$ with $q$ elements. Let $A = \mathbb{F}_q[T]$ denote the polynomial ring in the indeterminate $T$ over $\mathbb{F}_q$, and $K = \mathbb{F}_q(T)$ represent the rational function field. Define $K_\infty = \mathbb{F}_q((\pi_\infty))$ as the completion of $K$ at the infinite place, and $\mathcal{O}_\infty = \mathbb{F}_q[[\pi_\infty]]$ as its ring of integers, where $\pi_\infty := T^{-1}$. Let $|\cdot| = |\cdot|_\infty$ denote the normalized absolute value on $K_\infty$ with $|T|_\infty := q$. Let $\mathbb{C}_\infty$ denote the completion of an algebraic closure of $K_\infty$. Define $\Omega = \mathbb{C}_\infty - K_\infty$ as the Drinfeld upper half plane. Let $G$ be the group scheme $\GL(2)$ over $\mathbb{F}_q$, with $Z$ denoting the group scheme of scalar matrices in $G$. Consider a nonzero ideal $\n \lhd A$. The level-$\n$ Hecke congruence subgroup of $G(A)$ is defined as
$$\Gamma_0(\n):=\left\{\begin{pmatrix}
    a & b\\ c & d
\end{pmatrix}\in G(A)~\middle|~c\equiv 0~\text{mod}~\n\right\}.$$Let $\Gamma_0(\n)$ act on $\Omega$ via linear fractional transformations. Drinfeld proved in \cite{drinfeld_elliptic_1974} that the quotient $\Gamma_0(\n)\backslash \Omega$ is the space of $\mathbb{C}_\infty$-points of an affine curve $Y_0(\n)$ defined over $K$, representing a moduli space of rank-$2$ Drinfeld modules. The unique smooth projective curve over $K$ containing $Y_0(\n)$ as an open subvariety is denoted by $X_0(\n)$, termed as the Drinfeld modular curve of level $\n$. Let $J_0(\n)$ denote the Jacobian variety of $X_0(\n)$ and $\TT(\n):=J_0(\n)(K)_{\text{tors}}$ its rational torsion subgroup. According to the Lang-N\'eron theorem, $\TT(\n)$ is a finite abelian group. To investigate the structure of $\TT(\n)$, we study the rational cuspidal divisor class group $\CC(\n)$ of $X_0(\n)$, as defined in \cite{ho_rational_2024}. By Gekeler \cite[Theorem 1.2]{gekeler_note_2000}, $\CC(\n)$ is a finite group, leading to the inclusion $$\CC(\n) \subseteq \TT(\n).$$

The structure of $\CC(\n)$ is known in some special cases. First, when $\n = \p \lhd A$ is a prime ideal, the structure of $\CC(\p)$ was fully determined by Gekeler \cite[Corollary 3.23]{gekeler_1997}. Second, Papikian and Wei \cite[Theorem 1.3 and Proposition 1.5]{papikian_rational_2017} determined the prime-to-$(q-1)$ part of $\CC(\n)$ for any square-free ideal $\n\lhd A$. In particular, when $\n = \p\q$ is square-free with two prime ideals $\p$ and $\q$ of $A$, the structure of $\CC(\p\q)$ was entirely determined in \cite[Theorem 1.1]{papikian_eisenstein_2015}. Third, when $\n = \p^r\lhd A$ is a prime power ideal with $r\geq 2$, the structure of $\CC(\p^r)$ was completely determined in the following:
\begin{thm}[{\cite[Theorem 3.5]{ho_rational_2024}}]\label{The structure of C(p^r) in function field case}
Fix a prime ideal $\p\lhd A$ and $r\geq 2$. Then $$\CC(\p^r) \cong \left(\bigoplus_{1\leq i\leq m}\frac{\mathbb{Z}}{|\p|^{r-i}M(\p)\mathbb{Z}}\right)\oplus\left(\bigoplus_{m+1\leq i\leq r-2}\frac{\mathbb{Z}}{|\p|^{i}M(\p)\mathbb{Z}}\right)\oplus \frac{\mathbb{Z}}{M(\p)\mathbb{Z}}\oplus \frac{\mathbb{Z}}{N(\p)\mathbb{Z}},$$ where $m:=\lfloor \frac{r-1}{2}\rfloor$, $M(\p) := \frac{|\p|^2-1}{q^2-1}$, and $$N(\p):=\begin{cases}
\frac{|\p|-1}{q^2-1},  & \text{if $\deg(\p)$ is even.} \\
\frac{|\p|-1}{q-1},  & \text{otherwise.}
\end{cases}$$
Moreover, an explicit basis of $\CC(\p^r)$ is provided in \cite[Theorem 3.5]{ho_rational_2024}.
\end{thm}

\begin{rem}
When $\n\lhd A$ is square-free or $\deg(\n) = 3$, all the cusps of $X_0(\n)$ are rational by \cite[Proposition 6.7]{gekeler_invariants_2001} and \cite[Lemma 3.1]{papikian_eisenstein_2016}. Thus, the cuspidal divisor groups of $J_0(\n)$ defined in \cite[p. 522]{papikian_rational_2017} and \cite[p. 399]{papikian_eisenstein_2016} in the above cases are the same as the rational cuspidal divisor class groups of $X_0(\n)$ defined in \cite{ho_rational_2024}. However, this phenomenon does not universally hold. For instance, when $\n = (T^2+T+1)^2$ and $q=2$, the cuspidal divisor group of $J_0(\n)$ is strictly larger than the rational cuspidal divisor class group of $X_0(\n)$ by \cite[Example 8.8]{papikian_eisenstein_2016} and Theorem \ref{The structure of C(p^r) in function field case}.
\end{rem}

There is an analogue of the generalized Ogg's conjecture:
\begin{conj}[{\cite[Conjecture 1.2]{ho_rational_2024}}]
For any non-zero ideal $\n\lhd A$, $$\CC(\n) = \TT(\n).$$
\end{conj}
This conjecture remains open, and the following are some partial results:
\begin{enumerate}
\item P{\'a}l \cite[Theorem 1.4]{Pal2005} proved that $\CC(\p) = \TT(\p)$ for any prime ideal $\p \lhd A$.
\item Papikian and Wei \cite[Theorem 1.7]{papikian_eisenstein_2016} demonstrated that $\CC(\n) = \TT(\n)$ for an ideal $\n\lhd A$ generated by $T^3$ or $T^2(T-1)$.
\item Papikian and Wei \cite[Theorem 1.3]{papikian_rational_2017} showed that if $\n \lhd A$ is square-free, then $\CC(\n)_{\ell} = \TT(\n)_{\ell}$ for any prime $\ell$ not dividing $q(q-1)$.
\end{enumerate}

In this paper, we prove the following:
\begin{thm}[Main Theorem]\label{Main Theorem}
Let $\n=\p^r\lhd A$ be a prime power ideal with $r\geq 2$. For any prime $\ell\nmid q(q-1)$, we have $$\CC(\p^r)_{\ell} = \TT(\p^r)_{\ell}\cong\left(\frac{\mathbb{Z}_\ell}{M(\p)\mathbb{Z}_\ell}\right)^{r-1}\times\frac{\mathbb{Z}_\ell}{N(\p)\mathbb{Z}_\ell}.$$
\end{thm}

\begin{rem}
Both the square-free case discussed in \cite[Theorem 1.3]{papikian_rational_2017} and the prime power case in the Main Theorem utilize similar approaches in their proofs. However, a notable difference arises in the choice of tools: the former employs the Atkin-Lehner involutions $W_{\m}$ for studying Eisenstein ideals, while the latter relies on the Hecke operator $U_{\p}$ for the same purpose.
\end{rem}

We provide a brief overview of the proof for the Main Theorem. Fix a prime power ideal $\n=\p^r\lhd A$ with $r\geq 2$. The Hecke algebra $\mathbb{T}(\n)$ of level $\n$ is defined as the $\mathbb{Z}$-algebra generated by all Hecke operators $T_{\q}$, where $\q\lhd A$ are primes (including $\q = \p$), acting on the group $\mathcal{H}_0(\n, \mathbb{Z})$ of $\mathbb{Z}$-valued $\Gamma_0(\n)$-invariant cuspidal harmonic cochains on the Bruhat-Tits tree $\Tree$ of $\PGL(2,K_\infty)$. For further details, please refer to \cite[p. 385]{papikian_eisenstein_2016}, Section \ref{Section: Harmonic cochains and Fourier expansion}, and Section \ref{Section: Hecke operators and Eisenstein ideals}. The Eisenstein ideal $\mathfrak{E}(\n)$ of $\mathbb{T}(\n)$ is the ideal generated by the elements $$\{T_\q-|\q|-1\mid \q \text{ is prime}, \q\nmid \n\}.$$ Additionally, $\mathbb{T}(\n)$ naturally acts on the Jacobian $J_0(\n)$ of $X_0(\n)$; see Section \ref{Section: Two degeneracy maps} and \cite[p. 386]{papikian_eisenstein_2016} for more information. For the prime divisor $\p$ of $\n$, we usually write $U_\p := T_\p \in \mathbb{T}(\n)$. Consider the following exact sequences of groups:
\[\begin{tikzcd}
0\arrow{r} &\CC(\p^r)[U_\p]\arrow[hook]{r}\arrow[hook]{d} &\CC(\p^r)\arrow[two heads]{r}{U_{\p}}\arrow[hook]{d} &U_{\p}(\CC(\p^r))\arrow{r}\arrow[hook]{d} &0\\
0\arrow{r} &\TT(\p^r)[U_\p]\arrow[hook]{r} &\TT(\p^r)\arrow[two heads]{r}{U_{\p}} &U_{\p}(\TT(\p^r))\arrow{r} &0,
\end{tikzcd}\]
where $\CC(\p^r)[U_\p]$ and $\TT(\p^r)[U_\p]$ denote the kernels of $U_{\p}$. Considering the $\ell$-primary parts of the groups in the above diagram for a prime $\ell \nmid q(q-1)$, our objective is to demonstrate the following:
\begin{enumerate}
    \item $$U_{\p}(\CC(\p^r))_{\ell} = U_{\p}(\TT(\p^r))_{\ell}\cong\left(\frac{\mathbb{Z}_\ell}{M(\p)\mathbb{Z}_\ell}\right)^{r-2}\times\frac{\mathbb{Z}_\ell}{N(\p)\mathbb{Z}_\ell}.$$
    \item $$\CC(\p^r)_\ell[U_\p] = \TT(\p^r)_\ell[U_\p]\cong\frac{\mathbb{Z}_\ell}{M(\p)\mathbb{Z}_\ell}.$$
\end{enumerate}
Assuming the above and employing the short five lemma, we obtain the Main Theorem. The first statement is achieved through explicit computation of the $U_{\p}$ action on the cusps of $X_0(\p^r)$ combined with induction on $r$. To be specific, we decompose $U_{\p}$ into $U_{\p} = \alpha^{\ast}\circ \beta_{\ast}$, where $\alpha^{\ast}: J_0(\n/\p)\rightarrow J_0(\n)$ and $\beta_{\ast}: J_0(\n)\rightarrow J_0(\n/\p)$ are defined in Section \ref{Section: Two degeneracy maps}. In Lemma \ref{beta image}, we demonstrate that $\beta_{\ast}(\CC(\p^r))_{\ell} = \CC(\p^{r-1})_{\ell}$ for any prime $\ell\neq p$. Assume the induction hypothesis that $\CC(\p^{r-1})_{\ell} = \TT(\p^{r-1})_{\ell}$ for any prime $\ell\nmid q(q-1)$. Note that $\CC(\p) = \TT(\p)$ serves as the initial condition for the induction. Then for any prime $\ell\nmid q(q-1)$, $$U_{\p}(\CC(\p^r))_{\ell} = \alpha^{\ast}(\CC(\p^{r-1}))_{\ell} = \alpha^{\ast}(\TT(\p^{r-1}))_{\ell} \supseteq U_{\p}(\TT(\p^r))_{\ell}.$$ Now, the result follows from the natural inclusion $U_{\p}(\CC(\p^r))_{\ell} \subseteq U_{\p}(\TT(\p^r))_{\ell}$.

The second statement is established through an Atkin-Lehner type result; see Theorem \ref{Atkin-Lehner type result}. To be precise, for a coefficient ring $R$ (see Definition \ref{Def: coefficient ring}), let $\mathcal{E}(\n, R)[U_{\p}]$ represent the group of $R$-valued $\Gamma_0(\n)$-invariant harmonic cochains on $\Tree$ that are annihilated by both the actions of the Eisenstein ideal $\mathfrak{E}(\n)$ and the Hecke operator $U_{\p}$. We construct an explicit element $E_{\n}$ in $\mathcal{E}(\n, R)[U_{\p}]$; cf. Equation (\ref{eq: E_n}) and Lemma \ref{E_n is annihilated by E(n) and U_p}. Utilizing Theorem \ref{Atkin-Lehner type result}, we demonstrate that $E_{\n}$ generates $\mathcal{E}(\n, R)[U_{\p}]$ over $R$. Considering the values of $E_{\n}$ on specific edges of $\Tree$, we observe that the subgroup $\mathcal{E}_0(\n, R)[U_{\p}]$ of $\mathcal{E}(\n, R)[U_{\p}]$, consisting of cuspidal harmonic cochains, is contained in $R\left[M(\p)\right]\cdot E_{\n}$, where $R\left[M(\p)\right]:=\{c\in R\mid M(\p)\cdot c = 0\}$. Let $R := \mathbb{Z}/\ell^k\mathbb{Z}$ for a prime $\ell\nmid q(q-1)$ and a sufficiently large $k$. By \cite[Lemma 7.2]{papikian_eisenstein_2015}, there exists an injective map: 
\begin{equation} \label{eq: injection map}
\CC(\p^r)_\ell[U_\p]\subseteq\TT(\p^r)_\ell[U_\p]\hookrightarrow 
\mathcal{E}_{0}(\p^r, R)[U_{\p}] \subseteq R\left[M(\p)\right]\cdot E_{\n}\hookrightarrow \frac{\mathbb{Z}_\ell}{M(\p)\mathbb{Z}_\ell}.
\end{equation}
Moreover, according to Theorem \ref{The structure of C(p^r) in function field case} and Section \ref{Section: Two degeneracy maps}, we construct explicitly a cyclic subgroup of $\CC(\n)_\ell[U_\p]$ that is isomorphic to $\mathbb{Z}_\ell/M(\p)\mathbb{Z}_\ell$; see Equation (\ref{eq: C'}), Sequence (\ref{eq: main sequence}), and Proposition \ref{C(n)_l=T(n)_l}. Hence, we demonstrate that $\CC(\p^r)_\ell[U_\p] = \TT(\p^r)_\ell[U_\p]$.

\begin{rem}
The idea of considering Map (\ref{the canonical specialization map}) and Sequence (\ref{eq: injection map}) comes from the proof of $\CC(\p) = \TT(\p)$ for the prime level case in \cite[Theorem 1.4]{Pal2005}. More precisely, P{\'a}l showed that the image of the $\ell$-primary part of $\TT(\p)$, for a prime $\ell\neq p$, in the group $\Phi_{\infty}$ of connected components of the N{\'e}ron model of $J_0(\p)$ at $\infty$ with respect to specialization, injects into $\mathcal{E}_0(\p, \mathbb{Z}/\ell^k\mathbb{Z})$ for large enough $k$; cf. \cite[p. 133]{Pal2005}, which leads to \cite[Theorem 7.19]{Pal2005}.
\end{rem}

\begin{rem}
The primes $\ell$ dividing $q(q-1)$ lie beyond the scope of our current approach. Considering a non-zero ideal $\n\lhd A$, $J_0(\n)$ has bad reduction both at $\infty$ and at the primes dividing $\n$. In particular, $J_0(\n)$ has split toric reduction at $\infty$. In instances where $\ell$ divides $q(q-1)$, the specialization map $\TT(\n)_{\ell} \rightarrow (\Phi_{\infty})_{\ell}$ might not ensure injectivity; cf. \cite[Lemma 7.2]{papikian_eisenstein_2015}.
\end{rem}

The paper is outlined in the following. We begin in Section \ref{Section: The rational cuspidal divisor class group} with an introduction to the rational cuspidal divisor class group of $X_0(\n)$. Section \ref{Section: Harmonic cochains and Fourier expansion} covers harmonic cochains on the Bruhat-Tits tree $\Tree$ of $\PGL(2,K_\infty)$ and their Fourier expansions. Following this, Section \ref{Section: Hecke operators and Eisenstein ideals} discusses Hecke operators acting on the space of $R$-valued harmonic cochains on $\Tree$, where $R$ is a commutative ring with unity, alongside Eisenstein ideals for prime power levels. Section \ref{Section: Two degeneracy maps} explores the action of Hecke operators on the Jacobian $J_0(\n)$ of the Drifeld modular curve $X_0(\n)$. Finally, all introduced tools are integrated into the proof of the Main Theorem.

\section{The rational cuspidal divisor class group of $X_0(\n)$}\label{Section: The rational cuspidal divisor class group}
In this section, we recall from \cite{ho_rational_2024}. Let $\n\lhd A$ be a non-zero ideal. We also denote by $\n$ the monic polynomial in $A$ generating the ideal. The rational cuspidal divisor class group $\CC(\n)$ of $X_0(\n)$ is the subgroup of $J_0(\n)$ generated by the linear equivalence classes of the degree $0$ rational cuspidal divisors on $X_0(\n)$. By Gekeler \cite[(3.6)]{gekeler_1997}, the cusps of $X_0(\n)$ are in bijection with $\Gamma_0(\n)\backslash \mathbb{P}^1(K)$. Moreover, every cusp of $X_0(\n)$ has a representative $\begin{bmatrix} a \\ d \end{bmatrix}$ in $\Gamma_0(\n)\backslash \mathbb{P}^1(K)$, where $a,d\in A$ are monic, $d|\n$, and $\gcd(a,\n)=1$. Two representatives $\begin{bmatrix} a \\ d \end{bmatrix}$ and $\begin{bmatrix} a' \\ d' \end{bmatrix}$ represents the same cusp of $X_0(\n)$ if and only if $d=d'$ and $\alpha a' \equiv a\mod \gcd(d, \n/d)$ for some $\alpha \in \mathbb{F}_q^{\times}$ by \cite[Lemma 3.1 (i)]{papikian_eisenstein_2016}. Such a cusp is called of height $d$, independent of the choice of $a$. The cusps of $X_0(\n)$ of the same height are conjugate over $K$ by \cite[Section 6]{gekeler_invariants_2001} and \cite[Lemma 3.1 (iii)]{papikian_eisenstein_2016}. Denote by $[0]$ and $[\infty]$ the unique cusp of $X_0(\n)$ of height $1$ and $\n$, respectively.

In the following, we fix a prime power ideal $\n=\p^r\lhd A$ with $r\geq 1$. For a monic $d\mid \n$, denote by $(P(\n)_d)$ the sum of all the cusps of $X_0(\n)$ of height $d$, which is a rational cuspidal divisor on $X_0(\n)$; cf. \cite[Section 1.3]{ho_rational_2024}. Let $$\Div_{\cusp}^0(X_0(\n))(K):=\left\{C = \sum_{\substack{d|\n\\\text{monic}}}a_d\cdot (P(\n)_d)\bigm\vert \deg(C)=0,~a_d\in \mathbb{Z}\right\}$$ be the group of the degree $0$ rational cuspidal divisors on $X_0(\n)$. Recall that a modular function on $X_0(\n)$ is a meromorphic function on $\Omega \cup\ \mathbb{P}^1(K)$ which is invariant under the action of $\Gamma_0(\n)$. Moreover, a modular unit on $X_0(\n)$ is a modular function on $X_0(\n)$ that does not have zeros or poles on $\Omega$. Let $\mathcal{U}_\n$ be the subgroup of $\Div_{\cusp}^0(X_0(\n))(K)$ consisting of the divisors of modular units. Let $C_i := (P(\n)_{\p^i})-\deg(P(\n)_{\p^i})\cdot[\infty]\in \Div_{\cusp}^0(X_0(\n))(K)$, where $0\leq i \leq r-1$. Then $$\CC(\n):=\Div_{\cusp}^0(X_0(\n))(K)/\mathcal{U}_\n = \langle \overline{C_0},\overline{C_1},\cdots,\overline{C_{r-1}}\rangle.$$

Let $\Delta(z)$ be the Drinfeld discriminant function defined in \cite[p. 183]{gekeler_1997}. For $d|\n$, consider $\Delta_d(z):= \Delta(d z)$, which is a modular form on $\Omega$ of weight $q^2-1$ and type $0$ for $\Gamma_0(\n)$; cf. \cite[(1.2)]{gekeler_1997}. The zero orders of $\Delta_d(z)$ at the cusps of $X_0(\n)$ are defined in \cite[p. 47]{gekeler_drinfeld_1986}. Using \cite[Equations (3.10) and (3.11)]{gekeler_1997}, we found the divisor of $\Delta_d(z)$ on $X_0(\n)$ in \cite[Section 2.1]{ho_rational_2024}. Let $\mathcal{O}(X_0(\n))^\ast$ be the (multiplicative) group of modular units on $X_0(\n)$. Recall the following group homomorphism in \cite[Section 2.1]{ho_rational_2024}:
\[\begin{tikzcd}
g: &\Div_{\cusp}^0(X_0(\n))(K) \arrow[r] & \displaystyle\mathcal{O}(X_0(\n))^\ast\tens_{\mathbb{Z}}\mathbb{Q}\\
&C=\displaystyle\sum_{\substack{d|\n\\\text{monic}}}a_d\cdot(P(\n)_d) \arrow[r,mapsto] & \left(\displaystyle\prod_{\substack{d|\n\\\text{monic}}}\Delta_d^{r_d}\right)\otimes \frac{1}{(q-1)(|\p|^2-1)|\p|^{r-1}},
\end{tikzcd}\]
where $r_d\in \mathbb{Z}$ are defined by
$$\begin{bmatrix}
r_1\\
r_\p\\
\vdots\\
r_{\p^r}
\end{bmatrix} := (q-1)(|\p|^2-1)|\p|^{r-1}\cdot\Lambda(\n)^{-1}\cdot\begin{bmatrix}
a_1\\
a_\p\\
\vdots\\
a_{\p^r}
\end{bmatrix}.$$
Here, $\Lambda(\n)^{-1}\in \GL_{r+1}(\mathbb{Q})$ is a certain matrix defined in \cite[Section 2.1]{ho_rational_2024}. The map $g$ provides a connection between the rational cuspidal divisor class group $\CC(\n)$ and the modular units on $X_0(\n)$. More precisely, if $g(C) = f\otimes a$ for some $C\in \Div_{\cusp}^0(X_0(\n))(K)$, $f\in \mathcal{O}(X_0(\n))^\ast$, and $a\in \mathbb{Q}$, then we have $C = a\cdot \divisor(f)$.

\section{Harmonic cochains and Fourier expansions}\label{Section: Harmonic cochains and Fourier expansion}
Let $\Tree$ be the Bruhat-Tits tree of $\PGL(2,K_\infty)$; cf. \cite[(1.3)]{GekelerREVERSAT}. Recall that $G$ is the group scheme $\GL(2)$ over $\mathbb{F}_q$, with $Z$ the group scheme of scalar matrices in $G$. Let $\mathcal{K} = G(\mathcal{O}_\infty)$ with its Iwahori subgroup $\mathcal{I}$ defined by $$\mathcal{I} = \left\{\begin{pmatrix}
a & b\\
c & d
\end{pmatrix}
\in \mathcal{K}\bigm\vert c\equiv 0 \bmod \pi_\infty\right\}.$$ Let $V(\Tree) = G(K_\infty)/\mathcal{K}\cdot Z(K_\infty)$ and $E(\Tree) = G(K_\infty)/\mathcal{I}\cdot Z(K_\infty)$ be the sets of the vertices and the oriented edges of $\Tree$, respectively; cf. \cite[(1.3)]{GekelerREVERSAT}. The group $G(K_\infty)$ acts on $E(\Tree)$ by left multiplication. For a function $f$ on $E(\Tree)$ and $\gamma\in G(K_\infty)$, we define a function $f|\gamma$ on $E(\Tree)$ by $(f|\gamma)(e) := f(\gamma e)$. 
\begin{defn} (van der Put)
Let $R$ be a commutative ring with unity. An $R$-valued harmonic cochain on $\Tree$ is a function $f: E(\Tree)\rightarrow R$ that satisfies
\begin{enumerate}
    \item (alternating) $$f(e)+f(\overline{e})=0\text{ for all }e\in E(\Tree).$$
    \item (harmonic) $$\sum_{\substack{e\in E(\Tree)\\t(e)=v}}f(e)=0\text{ for all }v\in V(\Tree).$$
\end{enumerate}
Here, for $e\in E(\Tree)$, $t(e)$ is its terminus and $\overline{e}$ is its inversely oriented edge.
\end{defn}
We introduce some notation. Fix a non-zero ideal $\n\lhd A$. Denote by $\mathcal{H}(\Tree, R)$ the group of $R$-valued harmonic cochains on $\Tree$ and by $\mathcal{H}(\n, R):=\mathcal{H}(\Tree, R)^{\Gamma_0(\n)}$ its subgroup of $\Gamma_0(\n)$-invariant harmonic cochains, i.e., $f|\gamma = f$ for all $\gamma\in \Gamma_0(\n)$. For $f\in \mathcal{H}(\n, R)$, it naturally defines a function $f'$ on the quotient graph $\Gamma_0(\n)\backslash \Tree$. If $f'$ has compact support, then we call $f$ a cuspidal harmonic cochain for $\Gamma_0(\n)$. Let $\mathcal{H}_0(\n, R)$ be the subgroup of $\mathcal{H}(\n, R)$ consisting of cuspidal harmonic cochains for $\Gamma_0(\n)$, and denote by $\mathcal{H}_{00}(\n, R)$ the image of $\mathcal{H}_0(\n, \mathbb{Z})\otimes R$ in $\mathcal{H}_0(\n, R)$. Note that if $R$ is flat over $\mathbb{Z}$, then $\mathcal{H}_{0}(\n, R) = \mathcal{H}_{00}(\n, R)$.
\begin{defn}[P{\'a}l {\cite[Section 2]{Pal2005}}] \label{Def: coefficient ring}
A ring $R$ is a coefficient ring if $1/p\in R$ and $R$ is a quotient of a discrete valuation ring $\widetilde{R}$ which contains $p$-th roots of unity.
\end{defn}
In the following, we fix $R = \mathbb{C}$ or a coefficient ring. Let $f\in \mathcal{H}(\Tree, R)$ which is invariant under the action of $\Gamma_\infty:=\left\{\begin{pmatrix}
    a & b\\ 0 & d
\end{pmatrix}\in G(A)\right\}$. The constant Fourier coefficient of $f$ is the $R$-valued function $f^0$ on $\pi_{\infty}^\mathbb{Z}$ defined by
$$
f^0(\pi_{\infty}^k)=\begin{cases}
q^{1-k}\sum_{u\in (\pi_{\infty})/(\pi_{\infty}^k)}f\left(\begin{pmatrix}
\pi_{\infty}^k & u\\ 0 & 1
\end{pmatrix}\right) &\text{if $k\geq 1$};\\
f\left(\begin{pmatrix}
\pi_{\infty}^k & 0\\ 0 & 1
\end{pmatrix}\right) &\text{if $k\leq 1$}.
\end{cases}
$$
Let $\eta: K_\infty\rightarrow\mathbb{C}^{\times}$ be the character $$\eta: \sum a_i\pi_\infty^i\mapsto\eta_0(\Tr(a_1)),$$ where $\Tr: \mathbb{F}_q\rightarrow\mathbb{F}_p$ is the trace and $\eta_0$ is any non-trivial character of $\mathbb{F}_p$. For a divisor $\m$ on $K$, the $\m$-th Fourier coefficient $f^{\ast}(\m)$ of $f$ is $$f^{\ast}(\m) = q^{-1-\deg(\m)}\sum_{u\in(\pi_{\infty})/(\pi_{\infty}^{2+\deg(\m)})}f\left(\begin{pmatrix}
\pi_{\infty}^{2+\deg(\m)} & u\\ 0 & 1
\end{pmatrix}\right)\eta(-m u),$$ if $\m$ is non-negative; $f^{\ast}(\m)=0$, otherwise. Here, $m\in A$ is the monic polynomial such that $\m=\divisor(m)\cdot \infty^{\deg(\m)}$.
\begin{thm}[Gekeler {\cite[(2.8)]{gekeler_improper_1995}}] \label{Thm: Fourier expansion}
Let $f\in \mathcal{H}(\Tree, R)$ which is invariant under $\Gamma_\infty$. Then
$$f\left(\begin{pmatrix}
    \pi_\infty^k & y\\ 0 & 1
\end{pmatrix}\right) = f^0(\pi_\infty^k)+\sum_{\substack{0\neq m \in A\\ \deg(m)\leq k-2}}f^{\ast}(\divisor(m)\cdot \infty^{k-2})\cdot \eta(my).$$ In particular, $f$ is uniquely determined by the Fourier coefficients $f^0$ and $f^{\ast}$.
\end{thm}
Recall the following theorem:
\begin{thm}[van der Put {\cite[Proposition (1.1)]{VanderPut1981-1982}}] \label{thm: van der Put map}
Let $\mathcal{O}(\Omega)^\ast$ be the group of non-vanishing holomorphic rigid-analytic functions on $\Omega$. There is a canonical exact sequence of $G(K_\infty)$-modules $$0\rightarrow\mathbb{C}_\infty^\ast\rightarrow\mathcal{O}(\Omega)^\ast\overset{\widetilde{r}}{\rightarrow} \HarZ \rightarrow 0.$$
\end{thm}
The van der Put map $\widetilde{r}$ extends naturally to the map $$\widetilde{r}:\mathcal{O}(X_0(\n))^\ast\tens_{\mathbb{Z}}\mathbb{Q}\hookrightarrow\mathcal{O}(\Omega)^\ast\tens_{\mathbb{Z}}\mathbb{Q}\overset{\widetilde{r}\otimes 1}{\rightarrow}\HarZ\tens_{\mathbb{Z}}\mathbb{Q}\hookrightarrow\HarQ.$$

\begin{Ex}
Recall that $\Delta\in \mathcal{O}(\Omega)^\ast$ is the Drinfeld discriminant function. Consider the function $\widetilde{r}(\Delta)\in \mathcal{H}(\Tree, \mathbb{Z})\subset \mathcal{H}(\Tree, \mathbb{C})$, which is invariant under $\Gamma_\infty$. By \cite[(2.5) and Corollary 2.8]{gekeler_1997}, we have
\begin{equation}\label{eq: Fourier coeff. of Delta}
\begin{cases}
\widetilde{r}(\Delta)^0(\pi_\infty^k) = -(q-1)q^{1-k};\\
\widetilde{r}(\Delta)^{\ast}(1) = (q^2-1)(q-1)/q.
\end{cases}
\end{equation}
\end{Ex}

Recall the following lemma:
\begin{lem}[{\cite[Lemma 2.14]{papikian_rational_2017}}]
Given a non-zero ideal $\m\lhd A$ and $f\in \mathcal{H}(\Tree, R)$ which is invariant under $\Gamma_\infty$, define $$f|B_{\m}:=f|\begin{pmatrix}
\m & 0\\ 0 & 1
\end{pmatrix}\text{~with~}f|B_{\m}^{-1}:=f|\begin{pmatrix}
1 & 0\\ 0 & \m
\end{pmatrix}.$$ Then we have
$$
\begin{cases}
(f|B_\m)^0(\pi_{\infty}^k) = f^0(\pi_{\infty}^{k-\deg(\m)});\\
(f|B_\m)^{\ast}(\n) = f^{\ast}(\n/\m).
\end{cases}
$$
Note that $f^{\ast}(\n/\m) = 0$ if $\m\nmid \n$.
\end{lem}
We derive that:
\begin{lem} \label{lem: F. coeff. of Delta_d}
Fix a prime ideal $\p\lhd A$ and $i\geq 1$. Consider $\widetilde{r}(\Delta_{\p^i})\in \mathcal{H}(\Tree, \mathbb{Z})\subset \mathcal{H}(\Tree, \mathbb{C})$, which is invariant under $\Gamma_\infty$. Then we have
$$
\begin{cases}
    \widetilde{r}(\Delta_{\p^i})^0(\pi_{\infty}^k) = -(q-1)q^{1-k}|\p|^i;\\
    \widetilde{r}(\Delta_{\p^i})^\ast(1) = 0.
\end{cases}
$$
\end{lem}

\begin{proof}
By the previous lemma, we have $\widetilde{r}(\Delta_{\p^i})^0(\pi_{\infty}^k) = \widetilde{r}(\Delta)^0(\pi_{\infty}^{k-\deg(\p^i)})$ and $\widetilde{r}(\Delta_{\p^i})^\ast(1) = \widetilde{r}(\Delta)^\ast((\p^i)^{-1})$, which imply the result by Equation (\ref{eq: Fourier coeff. of Delta}).
\end{proof}

\section{Hecke operators and Eisenstein ideals}\label{Section: Hecke operators and Eisenstein ideals}
Fix a prime power ideal $\n=\p^r$ of $A$ with $r\geq 2$, and let $R$ be a commutative ring with unity. For a non-zero ideal $\m\lhd A$, we define an $R$-linear transformation of the space of $R$-valued functions on $E(\Tree)$ by $$f|T_\m := \sum f|\begin{pmatrix}
a & b\\ 0 & d
\end{pmatrix},$$ where the sum is over $a, b, d\in A$ such that $a$, $d$ are monic, $(ad)=\m$, $(a)+\n = A$, and $\deg(b)<\deg(d)$. This transformation is the $\m$-th Hecke operator; cf. \cite[Section 2.3]{papikian_eisenstein_2016}. In particular, for the prime divisor $\p$ of $\n$, we have $$f|U_\p := \sum_{\substack{b\in A\\ \deg(b)<\deg(\p)}} f|\begin{pmatrix}
1 & b\\ 0 & \p
\end{pmatrix} = f|T_\p.$$

\begin{rem}
The definition of $T_{\p} = U_{\p}$ above is different from that of $T_{\p} := U_{\p}+B_{\p}$ in \cite[(1.17)]{gekeler_1997}.
\end{rem}

\begin{prop}[{\cite[Proposition 2.10]{papikian_eisenstein_2015}}]\label{decomp. of Hecke operators}
The Hecke operators preserve the spaces $\mathcal{H}(\n,R)$ and $\mathcal{H}_0(\n,R)$, and satisfy the recursive formulas:
\begin{enumerate}
    \item $T_{\m\m'}=T_{\m}T_{\m'}$ if $\m+\m'=A$.
    \item $T_{\q^i}=T_{\q^{i-1}}T_{\q}-|\q|T_{\q^{i-2}}$ for any prime $\q\neq\p$ and $i\geq 2$.
    \item $T_{\p^i} = U_{\p}^i$ for $i\geq 1$.
\end{enumerate}
\end{prop}

\begin{Ex} \label{Ex: U_p action on harm. cochains}
Consider the action of $U_{\p}$ on $\mathcal{H}(\n, R)$. Let $f\in \mathcal{H}(\n, R)$. Since $\p^2\mid \n$, we have $h := f|U_{\p}\in \mathcal{H}(\n/\p, R)$ by \cite[Lemma 2.6]{papikian_eisenstein_2016}. Moreover, since $h$ is invariant under $\Gamma_{\infty}$, we have $(h|B_{\p})|U_{\p} = |\p|\cdot h$ by \cite[Lemma 2.23]{papikian_eisenstein_2015}. We obtain:
\[\begin{tikzcd}
U_{\p}: \mathcal{H}(\n, R) \arrow{r} & \mathcal{H}(\n/\p, R) \subset \mathcal{H}(\n, R) \\
f \arrow[mapsto]{r} & h := f|U_{\p}\\
h|B_{\p} \arrow[mapsto]{r} & {|\p|\cdot h}.
\end{tikzcd}\]
\end{Ex}

\begin{lem}[{\cite[Corollary 2.11]{gekeler_1997}}]\label{small lemma 1}
Let $\q\lhd A$ be a prime not equal to $\p$. Then 
\begin{enumerate}
    \item $\widetilde{r}(\Delta)|T_{\q} = (|\q|+1)\cdot\widetilde{r}(\Delta).$
    \item $\widetilde{r}(\Delta)|(U_{\p}+B_{\p}) = (|\p|+1)\cdot\widetilde{r}(\Delta).$ Alternatively, $$\widetilde{r}(\Delta)|U_{\p} = (|\p|+1)\cdot\widetilde{r}(\Delta) - \widetilde{r}(\Delta_{\p}).$$
\end{enumerate}
\end{lem}

From the above, we prove the following:
\begin{lem}\label{small lemma 2} Let $k\geq 1$ and $\q\lhd A$ be a prime not equal to $\p$. Then
\begin{enumerate}
    \item $\widetilde{r}(\Delta_{\p^k})|T_{\q} = (|\q|+1)\cdot\widetilde{r}(\Delta_{\p^k})$.
    \item $\widetilde{r}(\Delta_{\p^k})|U_{\p} = |\p|\cdot \widetilde{r}(\Delta_{\p^{k-1}})$.
\end{enumerate}
\end{lem}
\begin{proof}
First, we have 
\begin{align*}
&\widetilde{r}(\Delta_{\p^k})|T_{\q} = \widetilde{r}(\Delta)|\begin{pmatrix}
\p^k & 0\\ 0 & 1
\end{pmatrix}\begin{pmatrix}
\q & 0\\ 0 & 1
\end{pmatrix}+\sum_{\substack{b\in A\\ \deg(b)<\deg(\q)}} \widetilde{r}(\Delta)|\begin{pmatrix}
\p^k & 0\\ 0 & 1
\end{pmatrix}\begin{pmatrix}
1 & b\\ 0 & \q
\end{pmatrix}\\
&= \widetilde{r}(\Delta)|\begin{pmatrix}
\q & 0\\ 0 & 1
\end{pmatrix}\begin{pmatrix}
\p^k & 0\\ 0 & 1
\end{pmatrix}+\sum_{\substack{b\in A\\ \deg(b)<\deg(\q)}} \widetilde{r}(\Delta)|\begin{pmatrix}
1 & b \p^k\\ 0 & \q
\end{pmatrix}\begin{pmatrix}
\p^k & 0\\ 0 & 1
\end{pmatrix}\\
&= \widetilde{r}(\Delta)|\begin{pmatrix}
\q & 0\\ 0 & 1
\end{pmatrix}\begin{pmatrix}
\p^k & 0\\ 0 & 1
\end{pmatrix}+\sum_{\substack{b\in A\\ \deg(b)<\deg(\q)}} \widetilde{r}(\Delta)|\begin{pmatrix}
1 & b\\ 0 & \q
\end{pmatrix}\begin{pmatrix}
\p^k & 0\\ 0 & 1
\end{pmatrix}\\
&=(\widetilde{r}(\Delta)|T_{\q})|\begin{pmatrix}
\p^k & 0\\ 0 & 1
\end{pmatrix} = (|\q|+1)\cdot\widetilde{r}(\Delta)|\begin{pmatrix}
\p^k & 0\\ 0 & 1
\end{pmatrix} = (|\q|+1)\cdot\widetilde{r}(\Delta_{\p^k}).
\end{align*}
For the second statement, since $\widetilde{r}(\Delta_{\p^{k-1}})$ is $\Gamma_{\infty}$-invariant, we have
\begin{align*}
&\widetilde{r}(\Delta_{\p^k})|U_{\p} = \sum_{\substack{b\in A\\ \deg(b)<\deg(\p)}} \widetilde{r}(\Delta_{\p^{k-1}})|\begin{pmatrix}
\p & 0\\ 0 & 1
\end{pmatrix}\begin{pmatrix}
1 & b\\ 0 & \p
\end{pmatrix}\\
&= \sum_{\substack{b\in A\\ \deg(b)<\deg(\p)}} \widetilde{r}(\Delta_{\p^{k-1}})|\begin{pmatrix}
\p & 0\\ 0 & \p
\end{pmatrix}\begin{pmatrix}
1 & b\\ 0 & 1
\end{pmatrix} = |\p|\cdot \widetilde{r}(\Delta_{\p^{k-1}}).
\end{align*}
\end{proof}

Recall that $\mathbb{T}(\n)$ is the $\mathbb{Z}$-algebra generated by all Hecke operators $T_{\q}$, where $\q\lhd A$ are primes, acting on the group $\mathcal{H}_0(\n, \mathbb{Z})$, and the Eisenstein ideal $\mathfrak{E}(\n)$ of $\mathbb{T}(\n)$ is the ideal generated by the elements $$\{T_\q-|\q|-1\mid \q \text{ is prime}, \q\nmid \n\}.$$ To simplify the notation, we put
\begin{itemize}
    \item $\mathcal{E}(\n, R):=\mathcal{H}(\n, R)[\mathfrak{E}(\n)]\subset\mathcal{H}(\n, R)$,
    \item $\mathcal{E}_{0}(\n, R):=\mathcal{H}_{0}(\n, R)[\mathfrak{E}(\n)]\subset\mathcal{H}_{0}(\n, R)$,
    \item $\mathcal{E}_{00}(\n, R):=\mathcal{H}_{00}(\n, R)[\mathfrak{E}(\n)]\subset\mathcal{H}_{00}(\n, R)$,
\end{itemize}
the subgroups consisting of the elements annihilated by $\mathfrak{E}(\n)$.

Let $v$ be a place of $K$. Denote by $K_v$ the completion of $K$ at $v$, $\mathcal{O}_v$ the ring of integers of $K_v$, $\mathbb{F}_v$ the residue field of $\mathcal{O}_v$, and $\overline{\mathbb{F}_v}$ an algebraic closure of $\mathbb{F}_v$. Moreover, we denote by $K_v^{\text{nr}}$ and $\mathcal{O}_v^{\text{nr}}$ the maximal unramified extension of $K_v$ and $\mathcal{O}_v$, respectively. Let $\mathcal{J}$ denote the N\'eron model of $J_0(\n)$ over $\mathbb{P}^1_{\mathbb{F}_q}$. Let $\mathcal{J}^0$ denote the relative connected component of the identity of $\mathcal{J}$, that is, the largest open subscheme of $\mathcal{J}$ containing the identity section which has connected fibres. The group of connected components of $J_0(\n)$ at $v$ is $\Phi_v := \mathcal{J}_{\mathbb{F}_v}/\mathcal{J}_{\mathbb{F}_v}^0$. There is a natural map in \cite[p. 586]{papikian_eisenstein_2015}, called the canonical specialization map: 
\begin{equation} \label{the canonical specialization map}
\wp_v: J_0(\n)(K_v^{\text{nr}})=\mathcal{J}(\mathcal{O}_v^{\text{nr}})\rightarrow \mathcal{J}_{\mathbb{F}_v}(\overline{\mathbb{F}_v})\rightarrow \Phi_v.
\end{equation}
We remark that the Hecke algebra $\mathbb{T}(\n)$ naturally acts on $J_0(\n)$, which functorially extends to $\mathcal{J}$; hence, $\mathbb{T}(\n)$ also acts on $\Phi_v$; cf. \cite[p. 386]{papikian_eisenstein_2016}.

\section{Two degeneracy maps}\label{Section: Two degeneracy maps}
Fix a prime power ideal $\n=\p^r$ of $A$ with $r\geq 2$. Recall that the curve $Y_0(\n)$ is the generic fibre of the coarse moduli scheme for the functor which associates to an $A$-scheme $S$ the set of isomorphism classes of pairs $(\phi, C_{\n})$, where $\phi$ is a rank-$2$ Drinfeld module over $S$ with an $\n$-cyclic subgroup $C_{\n}$; for details, see \cite{drinfeld_elliptic_1974}, \cite{gekeler1986drinfeld}, and \cite{GekelerREVERSAT}. Let $\alpha_\p(\n), \beta_\p(\n): Y_0(\n)\rightarrow Y_0(\n/\p)$ be the degeneracy maps with moduli-theoretic interpretations: \begin{align*}
&\alpha := \alpha_\p(\n): (\phi, C_{\n})\mapsto (\phi, C_{\n/\p}),\\
&\beta := \beta_\p(\n): (\phi, C_{\n})\mapsto (\phi/C_\p, C_{\n}/C_{\p}),
\end{align*}
where $C_{\n/\p}$ and $C_{\p}$ are the subgroups of $C_{\n}$ of order $\n/\p$ and $\p$, respectively. These morphisms are proper, and hence uniquely extend to morphisms $\alpha, \beta: X_0(\n)\rightarrow X_0(\n/\p)$; cf. \cite[p. 584]{papikian_eisenstein_2015}. Moreover, by Picard functoriality, the maps $\alpha$ and $\beta$ induce two homomorphisms: $$\alpha^\ast := \alpha_\p(\n)^\ast: J_0(\n/\p)\rightarrow J_0(\n),~\beta^\ast := \beta_\p(\n)^\ast: J_0(\n/\p)\rightarrow J_0(\n),$$ with their duals denoted by $$\alpha_\ast := \alpha_\p(\n)_\ast: J_0(\n)\rightarrow J_0(\n/\p),~\beta_\ast := \beta_\p(\n)_\ast: J_0(\n)\rightarrow J_0(\n/\p).$$ We take the Hecke operator ``$U_\p$'' acting on $J_0(\n)$ by:
\[\begin{tikzcd}
U_{\p}: J_0(\n) \arrow{r}{\beta_\ast} & J_0(\n/\p) \arrow{r}{\alpha^\ast} & J_0(\n).
\end{tikzcd}\]
Note that $U_{\p}:=\alpha_\p(\n)^\ast\circ \beta_\p(\n)_\ast = \beta_\p(\n\p)_\ast \circ \alpha_\p(\n\p)^\ast$; cf. \cite[Remark 3.5]{YOO_torsion_2023}. The definition of $U_{\p}$ here is the dual version of that in \cite[p. 584]{papikian_eisenstein_2015}, which are conjugate to each other by the Atkin-Lehner involution $W_{\n}$; cf. \cite[p. 398]{papikian_eisenstein_2016}.

\begin{rem}
For a prime ideal $\q\lhd A$ not dividing $\n$, the Hecke operator $T_{\q}$ commutes with the Atkin–Lehner involution $W_{\n}$, so $T_\q$ coincides with its dual; cf. \cite[p. 584]{papikian_eisenstein_2015} and \cite[p. 398]{papikian_eisenstein_2016}.
\end{rem}

Denote $Y_0(\n)_{\mathbb{C}_\infty} := \Gamma_0(\n)\backslash \Omega$ with its completion $X_0(\n)_{\mathbb{C}_\infty} := \Gamma_0(\n)\backslash (\Omega \cup \mathbb{P}^1(K))$. For $z\in \Omega \cup \mathbb{P}^1(K)$, we have $\alpha, \beta: X_0(\n)_{\mathbb{C}_{\infty}}\rightarrow X_0(\n/\p)_{\mathbb{C}_{\infty}}$ with $$\begin{cases}
    \alpha(\Gamma_0(\n)\cdot z) = \Gamma_0(\n/\p)\cdot z; \\
    \beta(\Gamma_0(\n)\cdot z) = \Gamma_0(\n/\p)\cdot \begin{pmatrix}
        \p & 0\\ 0 & 1
    \end{pmatrix}\cdot z.
    \end{cases}$$
Then for a cusp $\Gamma_0(\n)\cdot \begin{bmatrix}
a\\ \p^{i}
\end{bmatrix}$ of $X_0(\n)_{\mathbb{C}_{\infty}}$ in the form defined in Section \ref{Section: The rational cuspidal divisor class group}, we have 
\begin{align*}
&\alpha(\Gamma_0(\n)\cdot \begin{bmatrix}
a\\ \p^{i}
\end{bmatrix}) = \begin{cases}
\Gamma_0(\n/\p)\cdot \begin{bmatrix}
a\\ \p^{i}
\end{bmatrix},  & \text{if $0\leq i \leq r-2$.} \\
\Gamma_0(\n/\p)\cdot [\infty],  & \text{if $i = r-1$ or $r$.}
\end{cases}\\
&\beta(\Gamma_0(\n)\cdot \begin{bmatrix}
a\\ \p^{i}
\end{bmatrix}) = \begin{cases}
\Gamma_0(\n/\p)\cdot [0],  & \text{if $i = 0$ or $1$.} \\
\Gamma_0(\n/\p)\cdot \begin{bmatrix}
a\\ \p^{i-1}
\end{bmatrix},  & \text{if $2\leq i \leq r$.}
\end{cases}
\end{align*}
Moreover, the degeneracy coverings $\alpha$ and $\beta$ have degree $|\p|$. From the above and \cite[Lemma 2.2]{ho_rational_2024}, $\alpha$ and $\beta$ also induce maps between the divisor groups:
\begin{align*}
&\alpha^\ast, \beta^\ast: \Div(X_0(\n/\p))\rightarrow\Div(X_0(\n));\\
&\alpha_\ast, \beta_\ast: \Div(X_0(\n))\rightarrow\Div(X_0(\n/\p)).
\end{align*}
Now, we consider $U_{\p} := \alpha^\ast\circ \beta_\ast: \Div(X_0(\n))\rightarrow \Div(X_0(\n))$. Recall that $(P(\n)_d)$ is the sum of all the cusps of $X_0(\n)$ of height $d$. We have:
\begin{lem} \label{explicit degeneracy map}
\begin{enumerate}
    \item For $(P(\n/\p)_{\p^j})\in \Div(X_0(\n/\p))$,
    $$\alpha^\ast(P(\n/\p)_{\p^j}) = \begin{cases}
    |\p|\cdot (P(\n)_{\p^j}),  & \text{if $0\leq j\leq \lfloor\frac{r-1}{2}\rfloor$.} \\
    (P(\n)_{\p^j}),  & \text{if $\lfloor\frac{r+1}{2}\rfloor \leq j \leq r-2$.} \\
    (q-1)\cdot (P(\n)_{\p^{r-1}})+(P(\n)_{\p^{r}}),  & \text{if $j = r-1$.}
    \end{cases}$$
    \item For $(P(\n)_{\p^j})\in \Div(X_0(\n))$,
    $$\beta_\ast(P(\n)_{\p^j}) = \begin{cases}
    (P(\n/\p)_{1}),  & \text{if $j = 0$.}\\
    \frac{|\p|-1}{q-1}\cdot (P(\n/\p)_{1}),  & \text{if $j = 1$.}\\
    |\p|\cdot (P(\n/\p)_{\p^{j-1}}),  & \text{if $2\leq j \leq \lfloor\frac{r}{2}\rfloor$.} \\
    (P(\n/\p)_{\p^{j-1}}),  & \text{if $\lfloor\frac{r}{2}\rfloor+1 \leq j \leq r$.}
    \end{cases}$$
\end{enumerate}
\end{lem}

\begin{proof}
In Example \ref{Ex: U_p action on harm. cochains}, we have the decomposition:
\[\begin{tikzcd}
U_{\p}: \mathcal{H}(\n, \mathbb{Z}) \arrow{r}{\beta_\ast} & \mathcal{H}(\n/\p, \mathbb{Z}) \arrow{r}{\alpha^{\ast}} & \mathcal{H}(\n, \mathbb{Z}) \\
f \arrow[mapsto]{r} & f|U_{\p}\\
& h \arrow[mapsto]{r} & h.
\end{tikzcd}\] The first statement of the lemma follows directly from \cite[Lemma 2.1]{ho_rational_2024}. By Lemma \ref{small lemma 2}, we have 
$\beta_\ast(\widetilde{r}(\Delta_{\p})):=\widetilde{r}(\Delta_{\p})|U_{\p} = |\p|\cdot \widetilde{r}(\Delta)$. Hence, we derive that $\beta_\ast$ maps $\divisor(\Delta_{\p})\in \Div(X_0(\n))$ to $|\p|\cdot\divisor(\Delta)\in \Div(X_0(\n/\p))$. From \cite[Section 2.1]{ho_rational_2024}, the divisor of $\Delta_{\p}$ on $X_0(\n)$ is $$\sum_{j=0}^r a_j\cdot (P(\n)_{\p^j}) := |\p|^{r-1}\cdot (P(\n)_1) + \sum_{j=1}^r \frac{q-1}{\rho(j)}|\p|^{\max\{r-2j+1, 1\}}\cdot (P(\n)_{\p^j}),$$ where $a_j\in \mathbb{N}$ and $$\rho(j)=\begin{cases}
1,  & \text{if $0<j<r$.} \\
q-1,  & \text{otherwise.}
\end{cases}$$ Additionally, the divisor of $\Delta^{|\p|}$ on $X_0(\n/\p)$ is $$\sum_{j=0}^{r-1} b_j\cdot (P(\n/\p)_{\p^j}) := |\p|\cdot (P(\n/\p)_{\p^{r-1}})+\sum_{j=0}^{r-2} \frac{q-1}{\rho(j)}|\p|^{\max\{r-2j, 1\}}\cdot (P(\n/\p)_{\p^j}),$$
where $b_j\in \mathbb{N}$. Now, the result for $2\leq j\leq r$ follows from $\beta_\ast(P(\n)_{\p^j}) = \frac{b_{j-1}}{a_j}\cdot(P(\n/\p)_{\p^{j-1}})$. For $j = 0$ or $1$, we have $\beta_\ast(P(\n)_1) = x\cdot (P(\n/\p)_1)$ and $\beta_{\ast}(P(\n)_{\p}) = y\cdot (P(\n/\p)_1)$ for some $x, y\in \mathbb{N}$. Consider $\widetilde{r}(\phi)\in \mathcal{H}(\n, \mathbb{Z})$ with $$\phi:=\left(\frac{\Delta}{\Delta_\p}\right)^{|\p|}\left(\frac{\Delta_{\p^2}}{\Delta_\p}\right)\in \mathcal{O}(X_0(\n))^\ast.$$ According to Lemma \ref{small lemma 1} and \ref{small lemma 2}, we have $$\beta_\ast(\widetilde{r}(\phi)):=\widetilde{r}(\phi)|U_{\p} = 0.$$ Moreover, referring to \cite[Section 2.1]{ho_rational_2024}, the divisor of $\phi$ on $X_0(\n)$ is given by:
$$(q-1)(|\p|^2-1)|\p|^{r-2}\cdot\left[\frac{|\p|-1}{q-1}\cdot(P(\n)_1) - (P(\n)_\p)\right].$$ Since $\beta_\ast(\divisor(\phi))=0$, we have $y = \frac{|\p|-1}{q-1}\cdot x$. Considering $a_0$, $a_1$, and $b_0$, we find $$|\p|^{r-1}\cdot \beta_{\ast}(P(\n)_1)+(q-1)|\p|^{r-1}\cdot \beta_{\ast}(P(\n)_\p) = |\p|^{r}\cdot (P(\n/\p)_1),$$ which implies that $x=1$ and $y=\frac{|\p|-1}{q-1}$.
\end{proof}

In the following, we prove that $\CC(\p^{r-1})/\beta_{\ast}(\CC(\p^r))$ is $p$-primary.
\begin{lem}[cf. {\cite[Lemma 3.4]{YOO_torsion_2023}}]\label{beta image}
Fix a prime power ideal $\n = \p^r\lhd A$ with $r\geq 2$. Let $\beta_{\ast}: \Div_{\cusp}^0(X_0(\p^r))(K)\rightarrow\Div_{\cusp}^0(X_0(\p^{r-1}))(K)$ be the map induced from Lemma \ref{explicit degeneracy map}. Then $$\Div_{\cusp}^0(X_0(\p^{r-1}))(K)/\beta_{\ast}(\Div_{\cusp}^0(X_0(\p^r))(K))\cong (\mathbb{Z}/|\p|\mathbb{Z})^{\lfloor \frac{r}{2}\rfloor-1}.$$ In particular, we have $$\CC(\p^{r-1})/\beta_{\ast}(\CC(\p^r))\hookrightarrow (\mathbb{Z}/|\p|\mathbb{Z})^{\lfloor \frac{r}{2}\rfloor-1}.$$
\end{lem}

\begin{proof}
Let $C_i := (P(\n)_{\p^i})-\deg(P(\n)_{\p^i})\cdot[\infty]\in \Div_{\cusp}^0(X_0(\n))(K)$ for $0\leq i \leq r-1$, and $C_i' := (P(\n/\p)_{\p^i})-\deg(P(\n/\p)_{\p^i})\cdot[\infty]\in \Div_{\cusp}^0(X_0(\n/\p))(K)$. Then $\Div_{\cusp}^0(X_0(\n))(K)$ is generated by $\{C_0,C_1,\cdots,C_{r-1}\}$. By Lemma \ref{explicit degeneracy map}, we have
\begin{align*}
&\beta_{\ast}(\Div_{\cusp}^0(X_0(\n))(K))\\
&=\langle\beta_{\ast}(C_0),\beta_{\ast}(C_1),\beta_{\ast}(C_2),\cdots,\beta_{\ast}(C_{r-1})\rangle\\
&= \langle C_0',\frac{|\p|-1}{q-1}C_0',|\p|C_1',\cdots,|\p|C_{\lfloor\frac{r}{2}\rfloor-1}', C_{\lfloor\frac{r}{2}\rfloor}',\cdots,C_{r-2}'\rangle\\
&\subset \Div_{\cusp}^0(X_0(\n/\p))(K)\\
&= \langle C_0',C_1',\cdots,C_{\lfloor \frac{r}{2}\rfloor-1}', C_{\lfloor \frac{r}{2}\rfloor}',\cdots,C_{r-2}'\rangle,
\end{align*}
which implies that $$\Div_{\cusp}^0(X_0(\n/\p))(K)/\beta_{\ast}(\Div_{\cusp}^0(X_0(\n))(K))\cong (\mathbb{Z}/|\p|\mathbb{Z})^{\lfloor \frac{r}{2}\rfloor-1}.$$
\end{proof}

\section{Proof of the Main Theorem}\label{section: main thm}
Fix a prime power ideal $\n=\p^r\lhd A$ with $r\geq 2$. We prove Theorem \ref{Main Theorem} by induction on $r$. It is known that $\CC(\p) = \TT(\p)$. Assume that $\CC(\p^{r-1})_{\ell} = \TT(\p^{r-1})_{\ell}$ for a prime $\ell\nmid q(q-1)$. We want to show that $\CC(\p^r)_{\ell} = \TT(\p^r)_{\ell}$. Consider the following diagram:
\[\begin{tikzcd}
0\arrow{r} &\CC(\p^r)_{\ell}[U_\p]\arrow[hook]{r}\arrow[hook]{d} &\CC(\p^r)_{\ell}\arrow[two heads]{r}{U_{\p}}\arrow[hook]{d} &U_{\p}(\CC(\p^r))_{\ell}\arrow{r}\arrow[hook]{d} &0\\
0\arrow{r} &\TT(\p^r)_{\ell}[U_\p]\arrow[hook]{r} &\TT(\p^r)_{\ell}\arrow[two heads]{r}{U_{\p}}\arrow{rd}[sloped, swap]{(\beta_{\p})_{\ast}} &U_{\p}(\TT(\p^r))_{\ell}\arrow{r} &0\\
&&&(\beta_{\p})_{\ast}(\TT(\p^{r}))_{\ell}\arrow{u}{\alpha_{\p}^{\ast}}\arrow[r, phantom, sloped, "\subseteq"] &\TT(\p^{r-1})_{\ell}.
\end{tikzcd}\]
We call $U_{\p}(\TT(\p^r))_{\ell}\subset \TT(\p^r)_{\ell}$ the old part of $\TT(\p^r)_{\ell}$. By Lemma \ref{beta image} and induction hypothesis, we have 
$$(\beta_{\p})_\ast(\CC(\p^r))_\ell = \CC(\p^{r-1})_\ell = \TT(\p^{r-1})_\ell\supseteq (\beta_{\p})_\ast(\TT(\p^r))_\ell.$$ Thus, we have $U_\p(\CC(\p^r))_\ell\supseteq U_{\p}(\TT(\p^r))_\ell$. Moreover, since $\CC(\p^r)\subseteq \TT(\p^r)$, we have $U_\p(\CC(\p^r))_\ell = U_{\p}(\TT(\p^r))_\ell$. Note that for any prime $\ell\neq p$,
$$U_\p(\CC(\p^r))_\ell\cong\CC(\p^{r-1})_\ell
\cong\left(\frac{\mathbb{Z}_\ell}{M(\p)\mathbb{Z}_\ell}\right)^{r-2}\times\frac{\mathbb{Z}_\ell}{N(\p)\mathbb{Z}_\ell},
$$
where $M(\p) := \frac{|\p|^2-1}{q^2-1}$ and 
$$N(\p) := \begin{cases}
\frac{|\p|-1}{q^2-1},  & \text{if $\deg(\p)$ is even.} \\
\frac{|\p|-1}{q-1},  & \text{otherwise.}
\end{cases}$$
Now, we look at the new part $\TT(\p^r)_{\ell}[U_\p]$ of $\TT(\p^r)_{\ell}$. Recall that $(P(\n)_{\p})$ is the sum of all the cusps of $X_0(\n)$ of height $\p$. Define $\overline{C'}\in \CC(\p^r)$ by 
\begin{equation} \label{eq: C'}
C' := \frac{|\p|-1}{q-1} \cdot [0] - (P(\n)_\p) \in \Div_{\cusp}^0(X_0(\n))(K).
\end{equation}
By Lemma \ref{explicit degeneracy map}, $\overline{C'}$ is annihilated by $U_\p$. Recall the map $g: \Div_{\cusp}^0(X_0(\n))(K) \rightarrow\displaystyle\mathcal{O}(X_0(\n))^\ast\tens_{\mathbb{Z}}\mathbb{Q}$ in Section \ref{Section: The rational cuspidal divisor class group}. We have 
\begin{equation}\label{eq: g(C')}
g(C')=\left(\left(\frac{\Delta}{\Delta_{\p}}\right)^{|\p|}\left(\frac{\Delta_{\p^2}}{\Delta_\p}\right)\right)\otimes \frac{1}{(q-1)(|\p|^2-1)|\p|^{r-2}}.
\end{equation}
From the above and \cite[Lemma 3.6]{ho_rational_2024}, the order of $\overline{C'}$ in $\CC(\p^r)$ divides $M(\p)|\p|^{r-2}$. Since $\frac{\Delta_{\p^2}}{\Delta_\p}$ does not have a $p$-th root in $\mathcal{O}(\Omega)^\ast$ by \cite[Corollary 3.18]{gekeler_1997}, we see that $|\p|^{r-2}$ divides $\ord(\overline{C'})$; cf. \cite[Lemma 3.3]{ho_rational_2024}. Let $e := \begin{pmatrix}
\pi_{\infty}^{2} & \pi_{\infty}\\
0 & 1
\end{pmatrix} \in E(\Tree)$; cf. \cite[Section 2.2]{ho_rational_2024}. By \cite[Lemma 2.4]{ho_rational_2024}, we compute that $$|\p|^{r-2}\cdot(\widetilde{r}\circ g(C'))(e) = -\frac{|\p|}{q}\cdot\frac{1}{M(\p)},$$ which implies that $M(\p)$ also divides $\ord(\overline{C'})$; cf. \cite[Lemma 3.3]{ho_rational_2024}. Then we conclude that $\ord(\overline{C'}) = M(\p)|\p|^{r-2}$.

\begin{rem}
Combining the above with Theorem \ref{The structure of C(p^r) in function field case} and Lemma \ref{beta image}, the prime-to-$p$ part of $\CC(\p^r)[U_\p]$ is cyclic and generated by $|\p|^{r-2}\cdot\overline{C'}$ of order $M(\p)$. In contrast, the $p$-part of $\CC(\p^r)[U_\p]$ may not be cyclic.
\end{rem}

Recall the following:
\begin{lem}[{\cite[Lemma 7.2]{papikian_eisenstein_2015}}] \label{injectivity of specialization map}
Suppose that $\ell$ is a prime not dividing $q(q-1)$. There is a natural injective homomorphism $\TT(\n)_\ell\hookrightarrow \mathcal{E}_{00}(\n, \mathbb{Z}/\ell^k\mathbb{Z})$ for any $k\in\mathbb{Z}_{\geq 0}$ with $\ell^k\geq \#(\Phi_{\infty, \ell})$.
\end{lem}
\begin{proof}
The idea in the proof of {\cite[Lemma 7.2]{papikian_eisenstein_2015}} is $$\TT(\n)_\ell\hookrightarrow (\Phi_{\infty})_\ell[\mathfrak{E}(\n)]\hookrightarrow \mathcal{E}_{00}(\n, \mathbb{Z}/\ell^k\mathbb{Z})$$ when $\ell^k\cdot (\Phi_{\infty})_\ell = 0$ and $\ell\nmid q(q-1)$.
\end{proof}
From the above lemma, we obtain an injective $\mathbb{T}(\n)$-equivariant map:
\begin{equation} \label{eq: main sequence}
\frac{\mathbb{Z}_\ell}{M(\p)\mathbb{Z}_\ell}\cong\langle\overline{C'}\rangle_\ell\subseteq \CC(\n)_\ell[U_\p]\subseteq \TT(\n)_\ell[U_\p]\hookrightarrow \mathcal{E}_{00}(\n, \mathbb{Z}/\ell^k\mathbb{Z})[U_\p]
\end{equation}
for a prime $\ell\nmid q(q-1)$ and large enough $k$. We claim that $$\mathcal{E}_{00}(\n, \mathbb{Z}/\ell^k\mathbb{Z})[U_\p]\hookrightarrow \frac{\mathbb{Z}_\ell}{M(\p)\mathbb{Z}_\ell}.$$ Assume the above, then we have $$\CC(\p^r)_\ell[U_\p] = \TT(\p^r)_\ell[U_\p],$$ which is the new part of $\TT(\p^r)_{\ell}$. Combining the new and the old parts of $\TT(\p^r)_\ell$, we obtain that $\CC(\p^r)_\ell = \TT(\p^r)_\ell$ for any prime $\ell\nmid q(q-1)$.

To demonstrate the claim, we consider 
\begin{equation} \label{eq: E_n}
E_{\n}:=\frac{1}{(q-1)(q^2-1)}\cdot \widetilde{r}\left(\left(\frac{\Delta}{\Delta_\p}\right)^{|\p|}\left(\frac{\Delta_{\p^2}}{\Delta_\p}\right)\right)\in \mathcal{H}(\n, \mathbb{Q}).
\end{equation}
Recall that the order of $\overline{C'}$ in $\CC(\p^r)$ is $M(\p)|\p|^{r-2}$. Then by Theorem \ref{thm: van der Put map} and Equation (\ref{eq: g(C')}), we have $E_{\n} = M(\p)|\p|^{r-2}\cdot(\widetilde{r}\circ g(C'))\in \mathcal{H}(\n, \mathbb{Z})$.

\begin{lem}\label{E_n is annihilated by E(n) and U_p}
$E_{\n}\in \mathcal{E}(\n,\mathbb{Z})[U_{\p}]$.
\end{lem}

\begin{proof}
By Lemma \ref{small lemma 1} and \ref{small lemma 2}, for any prime $\q\lhd A$ not equal to $\p$, we have $E_{\n}|T_{\q}=(|\q|+1)E_{\n}$ and $E_{\n}|U_{\p} = 0$. Hence, $E_{\n}$ is annihilated by both the actions of the Eisenstein ideal $\mathfrak{E}(\n)$ and the Hecke operator $U_{\p}$.
\end{proof}

From now on, we always assume that $R$ is a coefficient ring. By composing with the map $\mathbb{Z}\rightarrow R$ sending $z$ to $z\cdot 1$, we also consider $E_{\n}$ as an element of $\mathcal{E}(\n, R)[U_{\p}]$. We claim that $E_{\n}$ generates $\mathcal{E}(\n, R)[U_{\p}]$ over $R$. Before we prove the claim, we recall the following:
\begin{lem}[{\cite[Lemma 2.15]{papikian_rational_2017}}]\label{lemma: eigenfunction}
Assume that $f\in \mathcal{H}(\n,R)$ is an eigenfunction of all $T_{\q}$, $\q\nmid \n$; that is, $f|T_{\q}=\lambda_{\q}f$ for some $\lambda_{\q}\in R$. Then the Fourier coefficients $f^{\ast}(\m)$, with $\m$ coprime to $\n$, are uniquely determined by $f^{\ast}(1)$ and the eigenvalues $\lambda_{\q}$, $\q\nmid \n$.
\end{lem}
\begin{proof}
This follows from Proposition \ref{decomp. of Hecke operators} and the fact that $(f| T_{\m})^\ast(1) = |\m|f^\ast(\m)$; cf. \cite[Lemma 3.2]{Pal2005}.
\end{proof}
We provide an Atkin-Lehner type result similar to \cite[Theorem 2.18]{papikian_rational_2017}:
\begin{thm}\label{Atkin-Lehner type result}
Let $\n=\p^r$ with a prime $\p\lhd A$ and $r\geq 2$. Suppose that $f\in \mathcal{H}(\n,R)$ is such that $f^\ast(\m)=0$ unless $\p$ divides $\m$. Then there exists $h\in \mathcal{H}(\n/\p,R)$ such that $f = h|B_{\p}$.
\end{thm}

\begin{proof}
Let $h:=f|B_{\p}^{-1}$ and $$\Gamma_0(\n/\p,\p) := \left\{\begin{pmatrix}
    a & b\\ c & d
\end{pmatrix}\in G(A)\mid c\in \n/\p,~b\in \p\right\}.$$ By \cite[Lemma 2.13]{papikian_rational_2017}, $h$ is $\Gamma_0(\n/\p,\p)$-invariant and harmonic. Moreover, by Theorem \ref{Thm: Fourier expansion}, for any $e = \begin{pmatrix}
\pi_{\infty}^k & y\\ 0 & 1
\end{pmatrix}\in E(\Tree)$ and $a\in A$, we have \begin{align*}
&h\left(\begin{pmatrix}
    1&a\\0&1
\end{pmatrix}e\right) = f\left(\begin{pmatrix}
    \pi_{\infty}^{k+\deg \p} & y+\frac{a}{\p} \\ 0 & 1
\end{pmatrix}\right)\\
&=f^0(\pi_{\infty}^{k+\deg \p})+\sum_{\substack{0 \neq m\in A\\ \deg m\leq k+\deg \p-2}}f^{\ast}(\divisor(m)\infty^{k+\deg \p -2})\eta\left(my+\frac{ma}{\p}\right)\\
&=f^0(\pi_{\infty}^{k+\deg \p})+\sum_{\substack{0 \neq m\in A\\ \deg m\leq k+\deg \p-2}}f^{\ast}(\divisor(m)\infty^{k+\deg \p -2})\eta(my)=h(e),
\end{align*}
where the third equality comes from the assumption that $f^\ast(\m)=0$ unless $\p$ divides $\m$. Thus, $h$ is also $\Gamma_{\infty}$-invariant, which completes the proof since $\Gamma_0(\n/\p)$ is generated by $\Gamma_0(\n/\p,\p)$ and $\Gamma_{\infty}$.
\end{proof}
Now, we are able to prove the following:
\begin{prop}\label{H(n,Z_l)}
$\mathcal{E}(\n, R)[U_{\p}] = R\cdot E_{\n}$.
\end{prop}
\begin{proof}
Note that $R\cdot E_{\n}\subseteq \mathcal{E}(\n, R)[U_{\p}]$. It suffices to prove the reverse inclusion. By Equation (\ref{eq: Fourier coeff. of Delta}) and Lemma \ref{lem: F. coeff. of Delta_d}, we have $$E_{\n}^0(\pi_\infty^k) = 0\text{~and~}E_{\n}^{\ast}(1) = \frac{|\p|}{q}\in R^{\ast}.$$ Let $f\in \mathcal{E}(\n, R)[U_{\p}]$, then consider $$\Tilde{f} := f - f^{\ast}(1)\cdot\frac{q}{|\p|}\cdot E_{\n}\in \mathcal{E}(\n, R)[U_{\p}].$$ First, by Lemma \ref{lemma: eigenfunction} and Theorem \ref{Atkin-Lehner type result}, since $\Tilde{f}^{\ast}(1) = 0$, there exists $h\in \mathcal{H}(\n/\p,R)$ which is $\Gamma_{\infty}$-invariant such that $h|B_{\p} = \Tilde{f}$. Note that $$(h|B_{\p})|U_{\p}=\Tilde{f}|U_{\p}=0.$$ Second, by \cite[Lemma 2.23]{papikian_eisenstein_2015}, since $h$ is $\Gamma_{\infty}$-invariant, we have $$|\p|\cdot h = (h|B_{\p})|U_{\p} = 0.$$
Hence, $h=0$ and $\Tilde{f} = h|B_{\p} = 0$, which implies that $$f = f^{\ast}(1)\cdot\frac{q}{|\p|}E_{\n}\in R\cdot E_{\n}.$$
\end{proof}

Recall that 
$$\left\{\begin{pmatrix}
    \pi_\infty^{-j} & 0\\ 0 & 1
\end{pmatrix}\in E(\Tree)\mid j\geq 0\right\}$$ forms an end on the Bruhat-Tits tree $\Tree$ towards $\infty$; cf. \cite[pp. 184 - 185]{gekeler_1997}. Then consider the action of $\begin{pmatrix}
    1 & 0\\ \p & 1
\end{pmatrix}\in G(A)$ on $\Tree$. We obtain that
$$\left\{e_j:=\begin{pmatrix}
    1 & 0\\ \p & 1
\end{pmatrix}\begin{pmatrix}
    \pi_\infty^{-j} & 0\\ 0 & 1
\end{pmatrix}\in E(\Tree)\mid j\geq 0\right\}$$
forms another end on $\Tree$ corresponding to $\begin{bmatrix}
1\\\p
\end{bmatrix}=\begin{pmatrix}
1&0\\\p&1
\end{pmatrix}\cdot\begin{bmatrix}
1\\0
\end{bmatrix}\in \mathbb{P}^1(K_\infty)$. This implies that $\{\Gamma_0(\n)\cdot e_j\mid j\geq 0\}$ forms an end on the quotient graph $\Gamma_0(\n)\backslash\Tree$ corresponding to the cusp $\Gamma_0(\n)\cdot\begin{bmatrix}
1\\\p
\end{bmatrix}$ on $X_0(\n)$; cf. \cite[Definition 2.5]{papikian_rational_2017}. Before computing $\mathcal{E}_{0}(\n, R)[U_{\p}]$, we consider the following:
\begin{lem}[cf. {\cite[Lemma 2.4]{ho_rational_2024}}] \label{eval of widetilde{r}(Delta)}
Let $j\geq 0$. Then we have
\begin{enumerate}
    \item $\widetilde{r}(\Delta)(e_j) = -(q-1)(q^{j+1}-q-1)$.
    \item Assume that $k\geq 1$ and $j\geq \max\{(k-2)\deg(\p), 0\}$. Then $$\widetilde{r}(\Delta_{\p^k})(e_j) = -(q-1)(q^{j+1}|\p|^{2-k}-q-1).$$
\end{enumerate}
\end{lem}
\begin{proof}
For $k\geq 0$ and $j\geq 0$, we have
\begin{align*}
&\begin{pmatrix}
    \p^k\pi_\infty^{-j} & 0\\ \p\pi_\infty^{-j} & 1
\end{pmatrix}\begin{pmatrix}
    1 & \p^{1-k}\pi_\infty^{j+(2-k)\deg(\p)}\\ 0 & -(\p\pi_\infty^{\deg(\p)})^{2-k}
\end{pmatrix}\begin{pmatrix}
    \p^{-1}\pi_\infty^{j+1} & 0\\ 0 & \p^{-1}\pi_\infty^{j+1}
\end{pmatrix}\\
&=\begin{pmatrix}
    \p^{k-1}\pi_\infty & \pi_\infty^{j+1+(2-k)\deg(\p)}\\ \pi_\infty & 0
\end{pmatrix}=\begin{pmatrix}
    \pi_\infty^{j+1+(2-k)\deg(\p)} & \p^{k-1}\\ 0 & 1
\end{pmatrix}\begin{pmatrix}
    0 & 1\\ \pi_\infty & 0
\end{pmatrix}.
\end{align*}
Thus, $\begin{pmatrix}
    \p^k\pi_\infty^{-j} & 0\\ \p\pi_\infty^{-j} & 1
\end{pmatrix}=\begin{pmatrix}
    \pi_\infty^{j+1+(2-k)\deg(\p)} & \p^{k-1}\\ 0 & 1
\end{pmatrix}\begin{pmatrix}
    0 & 1\\ \pi_\infty & 0
\end{pmatrix}$ in $E(\Tree)$.

First, assume that $k = 0$ and $j\geq 0$, then 
\begin{align*}
&\begin{pmatrix}
    \pi_\infty^{j+1+2\deg(\p)} & \p^{-1}\\ 0 & 1
\end{pmatrix}\begin{pmatrix}
    0 & 1\\ \pi_\infty & 0
\end{pmatrix}\begin{pmatrix}
    \p\pi_\infty^{\deg(\p)} & \pi_\infty^{j+\deg(\p)}\\ 0 & -(\p\pi_\infty^{\deg(\p)})^{-1}
\end{pmatrix}\begin{pmatrix}
    \pi_\infty^{-\deg(\p)} & 0\\ 0 & \pi_\infty^{-\deg(\p)}
\end{pmatrix}\\
&=\begin{pmatrix}
    \pi_\infty & 0\\ \p\pi_\infty & \pi_\infty^{j+1}
\end{pmatrix}=\begin{pmatrix}
    0 & 1\\ 1 & \p
\end{pmatrix}\begin{pmatrix}
    \pi_\infty^{j+1} & 0\\ 0 & 1
\end{pmatrix}\begin{pmatrix}
    0 & 1\\ \pi_\infty & 0
\end{pmatrix}.
\end{align*}
This implies that
\begin{align*}
&\widetilde{r}(\Delta)(e_j)=\widetilde{r}(\Delta)(\begin{pmatrix}
    \pi_\infty^{-j} & 0\\ \p\pi_\infty^{-j} & 1
\end{pmatrix})\\
&=\widetilde{r}(\Delta)(\begin{pmatrix}
    \pi_\infty^{j+1+2\deg(\p)} & \p^{-1}\\ 0 & 1
\end{pmatrix}\begin{pmatrix}
    0 & 1\\ \pi_\infty & 0
\end{pmatrix})\\
&=\widetilde{r}(\Delta)(\begin{pmatrix}
    0 & 1\\ 1 & \p
\end{pmatrix}\begin{pmatrix}
    \pi_\infty^{j+1} & 0\\ 0 & 1
\end{pmatrix}\begin{pmatrix}
    0 & 1\\ \pi_\infty & 0
\end{pmatrix})\\
&= -\widetilde{r}(\Delta)(\begin{pmatrix}
    \pi_\infty^{j+1} & 0\\ 0 & 1
\end{pmatrix})=-(q-1)(q^{j+1}-q-1),
\end{align*}
where the penultimate equality comes from \cite[Equation (2.3)]{gekeler_1997}.

Second, assume that $k \geq 1$ and $j\geq \max\{(k-2)\deg(\p), 0\}$, then 
\begin{align*}
&\widetilde{r}(\Delta_{\p^k})(e_j)=\widetilde{r}(\Delta)(\begin{pmatrix}
    \p^k & 0\\ 0 & 1
\end{pmatrix}e_j)=\widetilde{r}(\Delta)(\begin{pmatrix}
    \p^k\pi_\infty^{-j} & 0\\ \p\pi_\infty^{-j} & 1
\end{pmatrix})\\
&=\widetilde{r}(\Delta)(\begin{pmatrix}
    \pi_\infty^{j+1+(2-k)\deg(\p)} & \p^{k-1}\\ 0 & 1
\end{pmatrix}\begin{pmatrix}
    0 & 1\\ \pi_\infty & 0
\end{pmatrix})\\
&=\widetilde{r}(\Delta)(\begin{pmatrix}
    1 & \p^{k-1}\\ 0 & 1
\end{pmatrix}\begin{pmatrix}
    \pi_\infty^{j+1+(2-k)\deg(\p)} & 0\\ 0 & 1
\end{pmatrix}\begin{pmatrix}
    0 & 1\\ \pi_\infty & 0
\end{pmatrix})\\
&= -\widetilde{r}(\Delta)(\begin{pmatrix}
    \pi_\infty^{j+1+(2-k)\deg(\p)} & 0\\ 0 & 1
\end{pmatrix})=-(q-1)(q^{j+1}|\p|^{2-k}-q-1),
\end{align*}
where the penultimate equality again comes from \cite[Equation (2.3)]{gekeler_1997}.
\end{proof}

\begin{lem}\label{Weak lemma}
Let $R\left[M(\p)\right]:=\{c\in R\mid M(\p)\cdot c = 0\}$. Then $$\mathcal{E}_{0}(\n, R)[U_{\p}] \subseteq R\left[M(\p)\right]\cdot E_{\n}.$$
\end{lem}

\begin{proof}
By Proposition \ref{H(n,Z_l)}, $\mathcal{E}_0(\n, R)[U_{\p}]$ is contained in $R\cdot E_{\n}$. Suppose that $c\cdot E_{\n}\in \mathcal{E}_0(\n, R)[U_{\p}]$ for some $c\in R$. Let $e_j:=\begin{pmatrix}
    1 & 0\\ \p & 1
\end{pmatrix}\begin{pmatrix}
    \pi_\infty^{-j} & 0\\ 0 & 1
\end{pmatrix}\in E(\Tree)$ with $j\geq 0$, then $E_{\n}(e_j) = q^{j+1}M(\p)$ by Lemma \ref{eval of widetilde{r}(Delta)}. Since $q\in R^{\ast}$, we have $c\cdot E_{\n}(e_j) = 0$ for any $j \gg 0$ if and only if $c\in R\left[M(\p)\right]$, which completes the proof.
\end{proof}

Now, we prove the following:
\begin{prop}\label{C(n)_l=T(n)_l}
Let $\n=\p^r\lhd A$ with a prime $\p\lhd A$ and $r\geq 2$. Assume that $\ell$ is a prime not dividing $q(q-1)$. Then $$C(\n)_{\ell}[U_{\p}]=\TT(\n)_{\ell}[U_{\p}].$$
\end{prop}

\begin{proof}
Let $R := \mathbb{Z}/\ell^k\mathbb{Z}$ with $k \gg 0$ such that $\TT(\n)_\ell\hookrightarrow \mathcal{E}_{00}(\n, R)$ in Lemma \ref{injectivity of specialization map}. By Lemma \ref{Weak lemma}, $$\mathcal{E}_{00}(\n, R)[U_\p]\subseteq \mathcal{E}_{0}(\n, R)[U_{\p}] \subseteq R\left[M(\p)\right]\cdot E_{\n}\hookrightarrow \frac{\mathbb{Z}_\ell}{M(\p)\mathbb{Z}_\ell}.$$ Thus, the equality holds in Sequence (\ref{eq: main sequence}). In particular, we have $$\CC(\n)_\ell[U_\p] = \TT(\n)_\ell[U_\p].$$
\end{proof}
From the above, we obtain our Main Theorem.

\section{Acknowledgments}
The author sincerely dedicates this paper to his supervisor, Prof. Mihran Papikian, for his invaluable guidance and encouragement. Special thanks are also extended to Prof. Fu-Tsun Wei for the insightful discussions on modular units. The author also expresses gratitude to Prof. Chieh-Yu Chang for the invitation to present this paper at NTHU and to Prof. Wen-Ching Winnie Li for the invitation to collaborate on a survey paper on Ogg's conjectures. Lastly, heartfelt appreciation goes to Chiu-Lin Huang for her companionship and emotional support.

\bibliographystyle{acm}
\bibliography{Bibliography}

\end{document}